\documentclass[12pt]{article}
 \usepackage[margin=1in]{geometry} 
\usepackage{amsmath,amsthm,amssymb,amsfonts}
 \usepackage{amssymb}

\numberwithin{equation}{section}

\usepackage{enumitem}
\usepackage[utf8]{inputenc}
\usepackage[pagewise]{lineno}

\usepackage[pagewise]{lineno}

\usepackage{hyperref}
 \usepackage{csquotes}
\usepackage[utf8]{inputenc}
\usepackage[english]{babel}
\usepackage{xcolor}

\providecommand{\keywords}[1]
{
  \small	
  \textbf{\textit{Keywords:}} #1
}

\newcommand{\MSC}[1]{%
  \small
  \textbf{\textit{Mathematics Subject Classification:}} #1
}
\title{Global boundedness in a Chemotaxis-May-Nowak model for virus dynamics with logistic damping}
\author{
    Minh Le\thanks{ Institute for Theoretical Sciences, Westlake University, China \texttt{(leminh@westlake.edu.cn)}} 
}
\date{}

\begin{document}
\maketitle

\begin{abstract}
This paper investigates the following May--Nowak type model for viral infection in a bounded domain $\Omega \subset \mathbb{R}^n$, $n \ge 2$:
\begin{equation*}
\begin{cases}
u_t = \Delta u - \chi \nabla \cdot (u \nabla v) + \kappa - u - uw - \mu \dfrac{u^{1+\alpha}}{\ln^k(u+e)}, \\[1mm]
\gamma v_t = \Delta v - v + uw, \\[1mm]
w_t = \Delta w - w + v,
\end{cases}
\end{equation*}
where $\chi \in \mathbb{R}$, $\mu > 0$, $k \in [0,1)$, $\alpha > 0$, and $\gamma \in \{0,1\}$. 

We establish the global existence and uniform-in-time boundedness of classical solutions for suitably regular initial data under one of the following conditions:
\begin{enumerate}[label=(\Alph*)]
\item $\gamma = 1$, $n = 2$, $k \in [0,1)$, $\alpha = 1$, and $\mu > 0$;
\item $\gamma = 1$, $3 \le n \le 5$, $k = 0$, $\alpha = 1$, and $\mu $  is sufficiently large;
\item $\gamma = 1$, $n \ge 6$, $k = 0$, $\alpha > \frac{n-2}{4}$, and $\mu > 0$;
\item $\gamma = 0$, $n \ge 2$, $k = 0$, $\alpha > \frac{n+2}{2}$, and $\mu>0$.
\end{enumerate}

In particular, in the physically relevant dimensions $n = 2,3$, both subquadratic and quadratic damping are sufficient to prevent blow-up. Moreover, in the fully parabolic case ($\gamma = 1$), our results improve upon recent findings by relaxing the condition $\alpha > \frac{n}{2}$ to weaker assumptions within the above parameter regimes.
\end{abstract}

\keywords{Chemotaxis,  global boundedness, May-Nowak, logistic source}\\
\MSC{35B35, 35K45, 35K55, 92C15, 92C17}

\numberwithin{equation}{section}
\newtheorem{theorem}{Theorem}[section]
\newtheorem{lemma}[theorem]{Lemma}
\newtheorem{remark}{Remark}[section]
\newtheorem{Prop}{Proposition}[section]
\newtheorem{Def}{Definition}[section]
\newtheorem{Corollary}{Corollary}[theorem]
\allowdisplaybreaks

\section{Introduction}

In the past three decades, the mathematical modeling of virus dynamics has received considerable attention, largely motivated by the seminal May–Nowak system proposed in \cite{BMSN}:
\begin{equation*}
    \begin{cases}
        u_t = \kappa -a_1 u- \lambda uw, &\qquad t>0, \\
        v_t = -a_2 v +\lambda uw, &\qquad t>0, \\
        w_t =-a_3 w+ \delta b, &\qquad t>0, 
    \end{cases}
\end{equation*}
where the three interacting variables $u=u(t)$, $v=v(t)$, and $w=w(t)$, respectively represent the densities of healthy uninfected immune cells, infected immune cells, and virus particles at time $t$. Healthy cells are generated by the body at a constant rate $\kappa$, perish at rate $a_1 u$, and become infected through contact with the virus at rate $\lambda uw$. Infected cells arise at rate $\lambda uw$ and are cleared at rate $a_2 v$. Virus particles are released by infected cells at rate $\gamma v$ and decay at rate $a_3 w$. A comprehensive qualitative analysis of the model can be found in \cite{BMSN,NB,May_Nowak}.

In 2013, a model describing the dynamics of HIV within a spatially heterogeneous environment was first proposed in \cite{SAMH} as follows:
\begin{equation} \label{2}
\begin{cases}
    u_t = D_u\Delta u -  \chi \nabla \cdot (u  \nabla v) + \kappa - a_1 u - \lambda uw, \qquad &\text{in } \Omega \times (0,\infty), \\
    v_t = D_v\Delta v - a_2 v + \lambda uw, \qquad &\text{in } \Omega \times (0,\infty), \\
    w_t = D_w\Delta w - a_3 w + \delta v, \qquad &\text{in } \Omega \times (0,\infty), \\
    \frac{\partial u}{\partial \nu} = \frac{\partial v}{\partial \nu} = \frac{\partial w}{\partial \nu} = 0, \qquad &\text{on } \partial \Omega \times (0,\infty), \\
    u(x,0) = u_0(x), \quad v(x,0) = v_0(x), \quad w(x,0) = w_0(x), \qquad &\text{in } \Omega,
\end{cases}    
\end{equation}
where $D_u$, $D_v$, and $D_w$ denote the respective diffusion coefficients, and $\Omega \subset \mathbb{R}^n$ for $n \geq 1$ is a bounded domain with smooth boundary. The chemotaxis term $-\chi \nabla \cdot (u \nabla v)$, with $\chi \in \mathbb{R}$, models the directed movement of target immune cells toward the gradient of cytokines produced at sites of infection.

The inclusion of the chemotactic term renders the analysis of the global dynamics of system \eqref{2} substantially more intricate, owing to the interplay between the nonlinear chemotactic response and the nonlinear infection term $uw$. We briefly survey some pertinent results concerning global solvability, long-time behavior, and blow-up phenomena for \eqref{2}. Global existence and boundedness of solutions to the fully parabolic system  were established in \cite{BPTW} under a smallness condition on $\chi$. Subsequently, Tao and Winkler \cite{Winkler_Tao} demonstrated the occurrence of blow-up for the parabolic--elliptic--parabolic variant in two and three spatial dimensions when $\chi$ is sufficiently large. In order to preclude blow-up, several strategies have been explored. One approach mitigates the chemotactic effect by replacing $\chi$ with $\chi(1+u)^{-\eta}$ for some $\eta>0$ (see \cite{Hu+Lankeit, Winkler_2019+, XLX-2021}); another weakens the nonlinear infection term by substituting $uw$ either with $\frac{uw}{1+u+w}$ (see \cite{Bellomo+Tao}) or with $u^\lambda w$ where $\lambda < \frac{2}{n}$ (see \cite{Mario}). A further strategy consists in enhancing the diffusion mechanism. Indeed, by replacing $\Delta u$ with $\nabla\cdot((u+1)^m \nabla u)$ for a sufficiently large $m$, it was shown in \cite{JM_2025, JY_2025} that system \eqref{2} admits a globally bounded classical solution.

It is well established that logistic damping can effectively suppress blow-up phenomena in the Keller--Segel system and in a variety of related chemotaxis models (see, e.g., \cite{Tello+Winkler, Winkler-logistic, Minh, Tian, Kurt, Kurt+Shen, Le-Kurt} and the references therein). A natural question therefore arises: whether the incorporation of a logistic damping term is sufficient to prevent blow-up in the present context. Motivated by this question, the present paper is devoted to investigating the global existence and boundedness of classical solutions to the following chemotaxis-May--Nowak system supplemented with a logistic source:
\begin{equation} \label{1}
\begin{cases}
    u_t = \Delta u -  \chi \nabla \cdot (u  \nabla v) + \kappa -u-uw- \mu\frac{u^{1+\alpha}}{\ln^k(u+e)}   ,  \qquad &\text{in } \Omega \times (0,\infty), \\
   \gamma v_t =  \Delta v -  v +uw, \qquad &\text{in } \Omega \times (0,\infty), \\
    w_t = \Delta w -w +v , \qquad &\text{in } \Omega \times (0,\infty), \\
    \frac{\partial u}{\partial \nu }=  \frac{\partial v}{\partial \nu }= \frac{\partial w}{\partial \nu } =0 \qquad &\text{on } \partial \Omega \times (0,\infty),\\
    u(x,0)=u_0(x), \quad v(x,0)=v_0(x), \quad w(x,0)=w_0(x) \qquad &\text{in } \Omega.
\end{cases}    
\end{equation}
where $\mu >0$, $\alpha>0$, $k \geq 0 $, $\chi \in \mathbb{R}$ and $\gamma \in \left \{0,1 \right \}$. 

We briefly recall several contributions aligned with this line of investigation. In \cite{JX_2023}, Wang and Si established that, in two spatial dimensions, the system \eqref{1} possesses a global generalized solution for any $\alpha>0$. Subsequently, Li and Zhang \cite{Li_Zhang} demonstrated that the logistic damping term guarantees the global existence and boundedness of classical solutions whenever $\alpha > \frac{n}{2}$ with $n \geq 2$. Furthermore, the long-time asymptotic behavior of solutions was also characterized in \cite{Li_Zhang}.

\textbf{Main results.} In comparison with the classical Keller--Segel system with a logistic source, the condition obtained in \cite{Li_Zhang} is far from optimal. Indeed, it is known that the value $\alpha=1$ plays a critical role in the dichotomy: if $\mu$ is sufficiently small, solutions may undergo finite-time blow-up \cite{Fuest_2021}, whereas if $\mu$ is sufficiently large, solutions exist globally and remain uniformly bounded \cite{Tello+Winkler, KA, Winkler-logistic}. In light of this dichotomy, a fundamental open question remains:
\textit{Does a quadratic logistic damping term prevent blow-up in system \eqref{1}?}
The primary objective of the present work is to address precisely this issue. More specifically, one of our main results establishes that, in two spatial dimensions, a subquadratic logistic damping of the form
\[
-\frac{\mu u^2}{\ln^k(u+e)}, \qquad \mu > 0, \quad k \in [0,1),
\]
is already sufficient to suppress blow-up phenomena. Moreover, in dimensions $n = 3,4,5$---including the physically most relevant case $n=3$---we prove that the quadratic damping term $-\mu u^2$ precludes blow-up, provided $\mu$ is taken sufficiently large.

Before stating our main results, we introduce the following assumptions. Throughout this work, we impose the following hypotheses on the initial data appearing in \eqref{1} and on the reproduction function:
\begin{equation} \label{initial}
        \begin{cases}
            u_0 \in W^{2, \infty}({\Omega}), \qquad u_0 \geq 0 &\text{ in }\Omega,\\
            \gamma v_0 \in W^{2, \infty}(\Omega), \qquad v_0  \geq 0 &\text{ in }\Omega,\\
              w_0 \in W^{2, \infty}(\Omega), \qquad w_0 \geq 0 &\text{ in }\Omega,\\
              \frac{\partial u_0}{\partial \nu }=  \gamma \frac{\partial v_0}{\partial \nu }= \frac{\partial w_0}{\partial \nu } =0 &\text{ on }\partial\Omega, 
        \end{cases}
    \end{equation}
and 
\begin{align} \label{kappa}
   \kappa \in C^1 \left ( \bar{\Omega}\times [0,\infty) \right )\cap L^\infty \left ( \Omega \times (0, \infty) \right ) \quad  \text{is nonnegative and} \quad \left \| \kappa  \right \|_{L^\infty(\Omega \times (0, \infty) )} \leq \kappa_*,
\end{align}
for some $\kappa_*>0$.

Our first main result, which concerns the global existence and boundedness of solutions to the fully parabolic system \eqref{1}, is stated below. 
\begin{theorem} \label{thm}
    Let $\Omega \subset \mathbb{R}^n$ with $n \geq 2$ be a bounded domain with smooth boundary, let $\gamma=1$, $\chi \in \mathbb{R}$, and assume that $\kappa$ satisfies \eqref{kappa}. Suppose that one of the following conditions holds:
    \begin{enumerate}[label=(\Alph*)]
        \item $n=2$, $k\in [0,1)$, $\alpha=1$, and $\mu>0$;
        \item $3 \leq n \leq 5$, $k=0$, $\alpha=1$, and $\mu>\mu_0$ for some $\mu_0>0$;
        \item $n \geq 6$, $k=0$, $\alpha> \frac{n-2}{4}$, and $\mu>0$.
    \end{enumerate}
    Then the system \eqref{1} with initial conditions \eqref{initial} possesses a unique global classical solution $(u,v,w)$ such that 
    \begin{equation*}
        \begin{cases}
            u \in C^0\big( \bar{\Omega}\times [0,\infty) \big) \cap C^{2,1}\big( \bar{\Omega}\times (0,\infty) \big), \\
            v \in C^0\big( \bar{\Omega}\times [0,\infty) \big) \cap C^{2,1}\big( \bar{\Omega}\times (0,\infty) \big), \\
            w \in C^0\big( \bar{\Omega}\times [0,\infty) \big) \cap C^{2,1}\big( \bar{\Omega}\times (0,\infty) \big),
        \end{cases}
    \end{equation*}
    and $u$, $v$, and $w$ are strictly positive in $\bar{\Omega}\times (0,\infty)$. Moreover, the solution is uniformly bounded in the sense that 
    \begin{align*}
        \sup_{t>0} \Big\{ \|u(\cdot,t)\|_{L^\infty(\Omega)} + \|v(\cdot,t)\|_{W^{1,\infty}(\Omega)} + \|w(\cdot,t)\|_{L^\infty(\Omega)} \Big\} < \infty.
    \end{align*}
\end{theorem}
\begin{remark}
The present work sharpens the results obtained in \cite{Li_Zhang}, where the authors established boundedness under the condition $\alpha > \frac{n}{2}$ with $n \geq 2$. In addition, our findings reveal that, in two spatial dimensions, even a sub-logistic damping mechanism is sufficient to preclude blow-up. This blow-up suppression effect in the two-dimensional setting is reminiscent of analogous phenomena observed in the Keller--Segel system (see \cite{ Tian4}).
\end{remark}

\begin{remark}
To the best of our knowledge, it remains an open question whether blow-up can occur when $\alpha < 1$, or when $\alpha = 1$ and $\mu$ is sufficiently small. Should this indeed be the case, our result would be optimal in dimensions $n = 2,3,4,5$.
\end{remark}

\begin{remark}
By employing arguments analogous to those in \cite[Theorem 1.2]{Li_Zhang}, one can establish the long-time asymptotic behavior of solutions to system \eqref{1}.
\end{remark}

For the parabolic--elliptic case, the question of whether the inclusion of a logistic source suffices to guarantee global existence and boundedness of solutions remains unresolved. Our second objective is to provide an answer to this question. More precisely, we establish the global well-posedness of the parabolic--elliptic--parabolic system \eqref{1} in the following theorem.
\begin{theorem} \label{thm2}
    Let $\Omega \subset \mathbb{R}^n$ with $n \geq 2$ be a bounded domain with smooth boundary, let $\gamma=0$, $\chi \in \mathbb{R}$, $k=0$, $\alpha>\frac{n+2}{2}$, $\mu>0$, and assume that $\kappa$ satisfies \eqref{kappa}. Then the system \eqref{1} with initial conditions \eqref{initial} possesses a unique global classical solution $(u,v,w)$ such that 
    \begin{equation*}
        \begin{cases}
            u \in C^0\big( \bar{\Omega}\times [0,\infty) \big) \cap C^{2,1}\big( \bar{\Omega}\times (0,\infty) \big), \\
            v \in C^{2,0}\big( \bar{\Omega}\times (0,\infty) \big), \\
            w \in C^0\big( \bar{\Omega}\times [0,\infty) \big) \cap C^{2,1}\big( \bar{\Omega}\times (0,\infty) \big),
        \end{cases}
    \end{equation*}
    and $u$, $v$, and $w$ are strictly positive in $\bar{\Omega}\times (0,\infty)$. Furthermore, the solution is uniformly bounded in the sense that 
    \begin{align*}
        \sup_{t>0} \Big\{ \|u(\cdot,t)\|_{L^\infty(\Omega)} + \|v(\cdot,t)\|_{W^{1,\infty}(\Omega)} + \|w(\cdot,t)\|_{L^\infty(\Omega)} \Big\} < \infty.
    \end{align*}
\end{theorem}

\begin{remark}
Due to the lack of a uniform-in-time $L^1$ bound for $v$, we obtain a larger threshold for $\alpha$ compared to in the fully parabolic case ($\gamma=1$).
\end{remark}
\textbf{Challenges and main ideas.} The principal difficulty lies in proving that the quadratic logistic damping source is capable of preventing blow-up in dimensions less than six. Indeed, compared to the classical Keller--Segel system, system \eqref{1} is considerably more intricate to analyze due to the presence of the nonlinear signal production term $uw$. Specifically, the nonlinear contribution arising from the chemotactic term behaves asymptotically like $\chi u^2 w$, whereas the logistic damping provides a counteracting term of the form $-\mu u^2$. Consequently, it appears that the sole viable route to establishing boundedness is to first show that $w$ remains uniformly bounded in time. Pursuing this line of reasoning, we prove in Lemmas \ref{v} and \ref{wL} that if
\begin{align*}
    \sup_{t \in (0,T_{\rm max}- \tau)} \int_t^{t+\tau} \int_\Omega u^{1+\alpha}(\cdot,s) \, ds < \infty,
\end{align*}
where $\tau = \min\{1, \frac{T_{\rm max}}{2}\}$ and $T_{\rm max} \in (0,\infty]$ is defined in Lemma \ref{local}, and $\alpha > \frac{n-2}{4}$, then $w$ is indeed uniformly bounded in time. The fact that quadratic damping precludes blow-up only in dimensions $n < 6$ is a direct consequence of the inequality $\frac{n-2}{4} < 1$, which holds precisely when $n < 6$. Under the established boundedness of $w$, system \eqref{1} behaves essentially like a standard Keller--Segel system, and therefore the quadratic logistic damping is sufficient to guarantee the boundedness of solutions.

\textbf{Plan of the paper.} In Section \ref{S2}, we establish the local existence of solutions to system \eqref{1}, together with several auxiliary inequalities that will be employed in the subsequent analysis. Section \ref{S3} is devoted to deriving a collection of \emph{a priori} estimates for solutions, which play a pivotal role in the proofs of the main results. In Section \ref{S4}, we obtain uniform $L^p$ bounds for solutions to the fully parabolic case ($\gamma = 1$) and complete the proof of Theorem \ref{thm}. Finally, in Section \ref{S5}, we establish uniform-in-time $L^p$ boundedness for solutions to the parabolic--elliptic--parabolic system, which is then employed to prove Theorem \ref{thm2}.
\section{Preliminaries} \label{S2}
In this section, we establish the local existence of solutions to system \eqref{1}, and collect several auxiliary results, including $L^p$ estimates for the heat equation and a number of useful inequalities. We begin with the local well-posedness result. Since its proof follows from standard arguments based on well-established techniques in \cite{DM} and \cite{Winkler-logistic}, we omit the details for the sake of brevity.
\begin{lemma} \label{local}
    Let $\Omega \subset \mathbb{R}^n$ with $n \geq 1$ be a bounded domain with smooth boundary, let $ \chi \in \mathbb{R}$, $\gamma \in \left \{ 0,1\right \}$, $\mu>0$, $\alpha>0$ and $k \geq 0$, and assume that $\kappa$ satisfies \eqref{kappa}. Then for any initial data $(u_0, \gamma v_0,w_0)$ satisfies \eqref{initial}, there exists $T_{\rm max}\in (0, \infty]$ and a uniquely determined functions 
    \begin{equation*}
        \begin{cases}
            u \in C^0\left ( \bar{\Omega}\times [0,T_{\rm max}) \right )\cap C^{2,1}\left ( \bar{\Omega}\times (0,T_{\rm max}) \right ), \\
            w \in C^0\left ( \bar{\Omega}\times [0,T_{\rm max}) \right )\cap C^{2,1}\left ( \bar{\Omega}\times (0,T_{\rm max}) \right ), 
        \end{cases}
    \end{equation*}
    and 
    \begin{equation*}
        \begin{cases}
            v \in C^0\left ( \bar{\Omega}\times [0,T_{\rm max}) \right )\cap C^{2,1}\left ( \bar{\Omega}\times (0,T_{\rm max}) \right ) \qquad &\text{if } \gamma=1,\\
            v \in C^{2,0}\left ( \bar{\Omega}\times (0,T_{\rm max}) \right ) \qquad &\text{if } \gamma=0,
        \end{cases}
    \end{equation*}
    such that $(u,v,w)$ solves \eqref{1} classically in $\bar{\Omega}\times [0,T_{\rm max})$. Moreover, $u>0$, $v>0$ and $w>0$ in $\bar{\Omega}\times (0,T_{\rm max})$ and 
    \begin{align} \label{local-1}
       \text{if } T_{\rm max}< \infty \text{ then } \limsup_{t \to T_{\rm max}} \left \{ \left \|u(\cdot,t) \right \|_{L^\infty(\Omega)}+  \gamma \left \|v(\cdot,t) \right \|_{W^{1,\infty}(\Omega)}+ \left \| w(\cdot,t) \right \|_{L^\infty(\Omega)}  \right \} = \infty.
    \end{align}
\end{lemma}

Henceforth, we denote by $(u,v,w)$ a solution of \eqref{1} in $\Omega \times (0,T_{\rm max})$ as provided by Lemma \ref{local}. We next recall a standard $L^p$ regularity result for parabolic equations, as established in \cite[Lemma 4.1]{DM}.
\begin{lemma} \label{C52.Para-Reg}
Let $\Omega \subset \mathbb{R}^n$ with $n \geq 2$ be a bounded domain with smooth boundary. Suppose that $p\geq 1$, $j \in \left \{1,2 \right \}$ and  
\begin{equation*}
    \begin{cases}
     q_j &\in \left [1, \frac{np}{n-jp} \right ),  \qquad \text{when } p< \frac{n}{j},\\
     q_j &\in \left [1, \infty \right )\qquad \text{when } p= \frac{n}{j},\\
     q_j & \in  \left [1, \infty \right ], \qquad \text{when } p>\frac{n}{j},
     \end{cases}
\end{equation*}
Assuming $V_0 \in W^{2,\infty}(\Omega)$ satisfying $\frac{\partial V_0}{\partial \nu}=0$ and $V$ is a classical solution to the following system
\begin{equation}\label{C52.parabolic-equation}
    \begin{cases}
     V_t = \Delta V  - a V + f &\text{in } \Omega \times (0,T), \\ 
\frac{\partial V}{\partial \nu} =  0 & \text{on }\partial \Omega \times (0,T),\\ 
 V(\cdot,0)=V_0   & \text{in } \Omega,
    \end{cases}
\end{equation}
where $f \in C(\Omega \times (0,T))$, $a>0$ and $T\in (0,\infty]$. If there exists $M>0$ satisfying 
\begin{align*}
    \left \| f(\cdot, t) \right \|_{L^p(\Omega)} \leq M \qquad \text{for all }t\in (0,T),
\end{align*}
then there exists $C_j>0$ depending on $p,q_j$ and $M$ such that 
\begin{align*}
    \left \| \nabla V(\cdot,t) \right \|_{q_1} \leq C_1, \qquad   \left \|  V(\cdot,t) \right \|_{q_2} \leq C_2 \qquad \text{for all }t\in (0,T_{\rm max}).
\end{align*}

\end{lemma}
The following lemma provides a key inequality for comparing \( u \) and \( \Delta v \). For a detailed proof, we refer interested readers to \cite{WMS}[Lemma 2.3].
\begin{lemma} \label{l1}
    Assuming that $\Omega \subset \mathbb{R}^n$ with $n \geq 1$,  $ p \in (n, \infty)$, and $T\in (0,\infty]$. Then there exists $C= C(n,p,\Omega)>0$ such that for any $t \in (t_0, T)$ with $t_0:= \min \left \{ 1, \frac{T}{2}\right \}$, the following holds
\begin{align}
    \int_{t_0}^t  e^{\frac{ps}{2}} \int_\Omega |\Delta g|^p\,dx  \, ds \leq C \left ( \int_{t_0}^t e^{\frac{ps}{2}} \int_\Omega |f|^p \,dx \, ds + e^{\frac{pt_0}{2}} \left \| \Delta g(\cdot, t_0) \right \|^p_{L^p(\Omega)}\right ),
\end{align}
for any $f \in L^p \left (  \Omega \times (0,T) \right ) $ and $g$ is a classical solution to the following system with initial condition $g_0 \in W^{2, p}(\Omega)$:
\begin{equation}
    \begin{cases}
     g_t = \Delta g  -  g + f &\text{in } \Omega \times (0,T), \\ 
\frac{\partial g}{\partial \nu} =  0 & \text{on }\partial \Omega \times (0,T),\\ 
 g(\cdot,0)=g_0   & \text{in } \Omega.
    \end{cases}
\end{equation}
\end{lemma}

We next present a generalized form of the Gagliardo--Nirenberg interpolation inequality, whose proof can be found in \cite{LiLankeit2016}[Lemma 2.3].
\begin{lemma} \label{GN}
Let $\Omega$ be a  bounded and smooth domain of $\mathbb{R}^n$ with $n \geq 1$. Let $r \geq 1$, $0< q\leq p < \infty$, $s>0$ such that 
\begin{align*}
    \frac{1}{r} \leq \frac{1}{n}+\frac{1}{p}.
\end{align*} Then there exists a constant $C_{GN}>0$ such that 
\begin{equation*}
    \left \| f \right \|^p_{L^p(\Omega)}\leq C_{GN}\left ( \left \| \nabla f \right \|_{L^r(\Omega)}^{pa}\left \| f \right \|^{p(1-a)}_{L^q(\Omega)} +\left \| f \right \|^p_{L^s(\Omega)} \right )\qquad \text{for all } f \in W^{1,r}(\Omega)\cap L^q(\Omega),
\end{equation*}
where $a= \frac{\frac{1}{q}-\frac{1}{p}}{\frac{1}{q}+\frac{1}{n}-\frac{1}{r}} \in [0,1]$.
\end{lemma}
We now recall an elementary ordinary differential inequality which will be later used in sequel sections.
\begin{lemma} \label{ODI}
    Let $T\in (0, \infty]$ and $\tau \in (0,T)$, $a>0$ and $b>0$, and assume that $y: [t_0,T) \to [0,\infty)$ for some $t_0 \in \mathbb{R}$ is absolutely continuous and such that 
    \begin{align*}
        y'(t)+ay(t) \leq h(t)\qquad \text{for a.e } t \in (t_0,T)
    \end{align*}
    with some nonnegative function $h \in L^1_{loc}([t_0,T))$ satisfying 
    \begin{align*}
        \frac{1}{\tau }\int_t^{t+\tau }h(s)\, ds \leq b \qquad \text{for all }t\in [t_0,T-\tau ). 
    \end{align*}
    Then 
    \begin{align*}
        y(t) \leq y(t_0)e^{-a(t-t_0)}+\frac{b \tau }{1- e^{-a\tau }} \qquad \text{for all }t\in [0,T).
    \end{align*}
\end{lemma}
\begin{proof}
     For a detailed proof, we refer the reader to \cite[ Lemma 3.4]{Winkler+2019}.
\end{proof}

The following lemma, which follows directly from \cite{Winkler_preprint}[Corollary 1.2], presents a modified form of the Gagliardo–Nirenberg interpolation inequality.

\begin{lemma}\label{C52.ILGN}
Let $\Omega \subset \mathbb{R}^2$ be a bounded domain with a smooth boundary, and suppose $m > 0$ and $\sigma > \xi \geq 0$. Then, for every $\varepsilon > 0$, there exists a constant $C = C(\varepsilon, \xi, \sigma) > 0$ such that the inequality  
\begin{align}\label{C52.ILGN.1}
    \int_\Omega w^{m+1} \ln^{\xi}(w+e) \,dx 
    &\leq \varepsilon \left( \int_\Omega w \ln^{\sigma}(w+e) \,dx \right) 
    \left( \int_\Omega |\nabla w^{\frac{m}{2}}|^2 \,dx \right) \notag \\
   & \quad + \varepsilon \left( \int_\Omega w \,dx \right)^m 
    \left( \int_\Omega w \ln^{\sigma}(w+e) \,dx \right) + C,
\end{align}
holds for any nonnegative function $w \in C^1(\bar{\Omega})$.
\end{lemma}
\begin{proof}
   For a detailed proof, we refer the reader to Lemmas 2.4 and 2.5 in \cite{Minh5}.
\end{proof}

\section{A Priori Estimate} \label{S3}
In this section, we present several useful estimates for solutions to \eqref{1}, ranging from a straightforward $L^1$ bound for $(u,v,w)$ to the more intricate $L^\infty$ bound for $w$. To prove that the quadratic logistic damping $-\mu u^2$ prevents blow-up when $\mu$ is sufficiently large in dimensions $n < 6$, we must carefully track the dependence of all constants on $\mu$. Accordingly, throughout this section, we explicitly indicate the dependence on $\mu$ for all constants that require specification. Let us begin with the following basic results, which are stated in the next two lemmas.   
\begin{lemma} \label{L1}
    Let $\gamma=1$, $n \geq 2$, $\alpha>0$, $k \in [0,1)$, $\chi \in \mathbb{R}$ and $\mu>0$. There exist $C_1>0$ independent of $\mu$ such that 
    \begin{align} \label{L1-1}
        \int_\Omega u(\cdot,t) +\int_\Omega v(\cdot,t) +\int_\Omega w(\cdot,t) \leq C_1 \qquad \text{for all }t\in (0,T_{\rm max}).
    \end{align}
    Moreover, there is $C_2>0$ which is independent of $\mu $ when $\mu>1$ such that 
    \begin{align} \label{L1-2}
        \int_t ^{t+\tau} \int_\Omega \frac{u^{1+\alpha}}{\ln^k(u+e)} \leq C_2 \qquad \text{for all }t\in (0,T_{\rm max}-\tau),
    \end{align}
    where $\tau = \min \left \{1, \frac{T_{\rm max}}{2} \right \}$.
\end{lemma}
\begin{proof}
    Integrating the first two equations of \eqref{1} over $\Omega$ and adding them together, we obtain that 
    \begin{align} \label{L1.1}
        \frac{d}{dt} \left \{ \int_\Omega u + \int_\Omega v \right \}+ \int_\Omega u + \int_\Omega v + \mu \int_\Omega \frac{u^{1+\alpha}}{\ln^k(u+e)} \leq \int_\Omega \kappa \qquad \text{for all }t\in (0,T_{\rm max}).
    \end{align}
    By applying Gronwall inequality, it follows that 
    \begin{align} \label{L1.2}
        \int_\Omega u(\cdot,t) + \int_\Omega v(\cdot,t) \leq c_1 \qquad \text{for all }t\in (0,T_{\rm max}),
    \end{align}
    where $c_1= \max \left \{ \int_\Omega u_0 +\int_\Omega v_0, \kappa_* |\Omega| \right \}$. From the third equation of \eqref{1}, we infer that 
    \begin{align*}
        \frac{d}{dt}\int_\Omega w + \int_\Omega w \leq c_1 \qquad \text{for all }t\in (0,T_{\rm max}).
    \end{align*}
    This, together with Gronwall inequality implies that 
    \begin{align*}
        \int_\Omega w(\cdot,t)  \leq c_2 \qquad \text{for all }t\in (0,T_{\rm max}),
    \end{align*}
    where $c_2= \max \left \{\int_\Omega w_0, c_1 \right \}$. Now, integrating \eqref{L1.1} over $(t,t+\tau)$ with   $\tau = \min \left \{1, \frac{T_{\rm max}}{2} \right \}$ yields
    \begin{align*}
        \mu \int_t ^{t+\tau}\int_\Omega \frac{u^{1+\alpha}}{\ln^k(u+e)} &\leq \kappa_*|\Omega| \tau + \int_\Omega u(\cdot,t) +\int_\Omega v(\cdot,t) \notag 
        \\
        &\leq c_1+\kappa_*|\Omega|  \qquad \text{for all }t\in (0,T_{\rm max}-\tau).
    \end{align*}
    This proves \eqref{L1-2} and completes the proof.
\end{proof}

\begin{lemma} \label{L1'}
     Let $\gamma=0$, $n \geq 2$, $\alpha>0$, $k \in [0,1)$, $\chi \in \mathbb{R}$ and $\mu>0$. There exist $C>0$ such that 
    \begin{align*} 
        \int_\Omega u(\cdot,t)  \leq C \qquad \text{for all }t\in (0,T_{\rm max})
    \end{align*}
    and 
     \begin{align*} 
        \int_\Omega w(\cdot,t)  \leq C \qquad \text{for all }t\in (0,T_{\rm max}).
    \end{align*}
\end{lemma}
\begin{proof}
    Integrating the first equation of \eqref{1} over $\Omega$ yields
    \begin{align} \label{L1'.1}
        \frac{d}{dt}\int_\Omega u +\int_\Omega u +\int_\Omega uw &\leq   \int_\Omega \kappa      \notag \\
        &\leq \kappa_* |\Omega| \qquad \text{for all }t\in (0,T_{\rm max}).
    \end{align}
    Employing Gronwall inequality to \eqref{L1'.1}, we obtain that 
    \begin{align}
        \int_\Omega u(\cdot,t) \leq  c_1 \qquad \text{for all }t\in (0,T_{\rm max}),
    \end{align}
    where $c_1=\left \{\kappa_* |\Omega|, \int_\Omega u_0 \right \}$. Now, integrating \eqref{L1'.1} over $(t, t+\tau)$ where $\tau = \min \left \{1, \frac{T_{\rm max}}{2} \right \}$, we deduce that 
    \begin{align} \label{L1'.2}
        \int_t^{t+\tau } \int_\Omega u(\cdot,s) w(\cdot,s)\, ds &\leq \kappa_* |\Omega| \tau + \int_\Omega u(\cdot,t) \notag \\
        &\leq c_2 \qquad  \text{for all }t\in (0,T_{\rm max}-\tau ),
    \end{align}
    where $c_2 = \kappa_* |\Omega| \tau+c_1$. Finally, we integrate the third equation of \eqref{1} over $\Omega$ to obtain
    \begin{align*}
        \frac{d}{dt} \int_\Omega w+ \int_\Omega w = \int_\Omega v \qquad \text{for all }t\in (0,T_{\rm max}),
    \end{align*}
    Using this and \eqref{L1'.2}, we infer that there exists $c_3>0$ such that 
    \begin{align*} 
        \int_\Omega w(\cdot,t)  \leq c_3 \qquad \text{for all }t\in (0,T_{\rm max}).
    \end{align*}
    The proof is now complete.
\end{proof}

The main objective of this section is to derive a uniform-in-time $L^\infty$ bound for $w$, which is far from obvious. To accomplish this, we employ an iteration argument; accordingly, we require the following lemma, which provides a key estimate for initiating the iterative procedure.

\begin{lemma} \label{w}
 Let $\gamma \in \left \{0,1 \right \}$, $n \geq 2$, $\alpha>0$, $k \in [0,1)$, $\chi \in \mathbb{R}$ and $\mu>0$.   Let $p_0 \in [ 1, \infty)$  and assume that there exists $M>0$ independent of $\mu$ if $\mu>1$ such that  
   \begin{align} \label{w-1}
       \int_\Omega w^{p_0}(\cdot,t) \leq M \qquad \text{for all }t\in (0,T_{\rm max}). 
   \end{align}
   Then for any $p>1$, there exist positive constants $c_1$, $c_2$ and $c_3$ independent of $\mu$ if $\mu>1$ such that 
    \begin{align*}
        \frac{d}{dt} \int_\Omega w^{p} +p \int_\Omega w^p +c_1 \int_\Omega w^{p+\frac{2p_0}{n}} \leq c_2 \int_\Omega v^{\frac{p+\frac{2p_0}{n}}{\frac{2p_0}{n}+1}} +c_3 \qquad \text{for all } t\in (0,T_{\rm max}).
    \end{align*}
\end{lemma}
\begin{proof}
    Applying Lemma \ref{GN} and using \eqref{w-1}, there exist $c_1>0$ independent of $\mu$ and $c_2>0$ independent of $\mu $ if $\mu>1$ satisfying 
    \begin{align} \label{w.1}
        \int_\Omega w^{p+ \frac{2p_0}{n}} &= \left \| w^{\frac{p}{2}} \right \|_{L^{2+ \frac{4p_0}{np}}(\Omega)}^{2+ \frac{4p_0}{np}} \notag \\
        &\leq \left ( c_1 \left \| \nabla w^{\frac{p}{2}} \right \|_{L^2(\Omega)}^a  \left \| w^{\frac{p}{2}} \right \|_{L^{\frac{2p_0}{p}}}^{1-a} + c_1\left \| w^{\frac{p}{2}} \right \|_{L^{\frac{2p_0}{p}}} \right )^{2+ \frac{4p_0}{np}} \notag \\
        &\leq \left ( c_2 \left \| \nabla w^{\frac{p}{2}} \right \|_{L^2(\Omega)}^a  + c_2 \right )^{2+ \frac{4p_0}{np}} ,
    \end{align}
    where 
    \begin{align*}
        a= \frac{\frac{p}{2p_0}- \frac{np}{2np+4p_0}}{\frac{p}{2p_0}+\frac{1}{n}-\frac{1}{2}} \in (0,1).
    \end{align*}
    Noting that 
    \begin{align*}
        \frac{2np+4p_0}{np} \cdot a =2, 
    \end{align*}
    and using \eqref{w.1}, there exists $c_3>0$ independent of $\mu$ if $\mu>1$ such that 
    \begin{align} \label{w.2}
        \int_\Omega w^{p-2}|\nabla w|^2 \geq c_3 \int_\Omega w^{p+\frac{2p_0}{n}} -c_3. 
    \end{align}
    Multiplying the third equation of \eqref{1} by $w^{p-1}$ with $p>1$, employing \eqref{w.2}, and using Young's inequality, we can find $c_4>0$ and $c_5>0$ independent of $\mu$ if $\mu>1$ such that
    \begin{align}
        \frac{d}{dt} \int_\Omega w^p+p \int_\Omega w^p  &=-p(p-1)\int_\Omega w^{p-2}|\nabla w|^2+ p\int_\Omega v w^{p-1}  \notag\\
        &\leq -c_4 \int_\Omega w^{p+\frac{2p_0}{n}} + p\int_\Omega v w^{p-1} +c_4 \notag \\
        &\leq -\frac{c_4}{2}\int_\Omega w^{p+\frac{2p_0}{n}} +c_5 \int_\Omega v^{\frac{p+\frac{2p_0}{n}}{\frac{2p_0}{n}+1}} +c_4 \qquad \text{for all }t\in (0,T_{\rm max}),
    \end{align}
    which finishes the proof.
\end{proof}

The next lemma enables us to improve the regularity of $v$, which constitutes the second ingredient for setting up the iteration process aimed at deriving an $L^\infty$ estimate for $w$; this estimate will be subsequently established in Lemma \ref{wL}.

\begin{lemma} \label{v}
    Let $\gamma=1$, $\alpha > \frac{n-2}{4}$ with $n \geq 2$, $k=0$, $\chi \in \mathbb{R}$ and $\mu>0$. Assume that there exists $M>0$ independent of $\mu$ if $\mu>1$ such that 
    \begin{align} \label{v-1}
        \int_\Omega v^{p_0}(\cdot,t) \leq M \qquad \text{for all }t\in (0,T_{\rm max}).
    \end{align}
    for some $p_0 \in \left [1,  \frac{n}{2}\right ]$. Then for any $1 \leq p< \frac{4\alpha+2}{n}p_0$, there exists $C>0$ independent of $\mu $ if $\mu>1$ such that 
    \begin{align*}
        \int_\Omega v^p(\cdot,t) \leq C \qquad \text{for all }t\in (0,T_{\rm max}).
    \end{align*}
\end{lemma}

\begin{proof}
   For any $p \in \left [1, \frac{4\alpha+2}{n}p_0 \right ) $, we can find $q \in \left [ 1, \frac{np_0}{n-2p_0} \right )$ for $p_0<\frac{n}{2}$ and $q \in [1, \infty)$ for $p_0=\frac{n}{2}$  such that 
    \begin{align} \label{v.1'}
        p= \frac{2p_0 \alpha}{n} + \frac{2q(\alpha+1)}{n+2q}.
    \end{align}
   By Lemma \ref{C52.Para-Reg} and \eqref{v-1}, there exists $c_1>0$ independent of $\mu$ if $\mu>1$ such that 
    \begin{align}\label{v.1}
        \int_\Omega w^{q}(\cdot,t) \leq c_1 \qquad \text{for all }t\in (0,T_{\rm max}).
    \end{align}
   Letting 
    \begin{align} \label{v.2}
        r= \frac{n+2q}{n}\cdot \left ( p + \frac{2p_0}{n} \right )-\frac{2q}{n}.
    \end{align}
    Applying Lemma \ref{w}, there exist $c_2>0$ and $c_3>0$ independent of $\mu$ if $\mu>1$ such that
    \begin{align} \label{v.3}
        \frac{d}{dt}\int_\Omega w^r + r \int_\Omega w^r + c_2 \int_\Omega w^{r+ \frac{2q}{n}} \leq c_3 \int_\Omega v^{ \frac{n}{n+2q} \cdot \left (r+\frac{2q}{n} \right )}+c_3 \qquad \text{for all }t\in (0,T_{\rm max}).
    \end{align}
    By employing the same argument as \eqref{w.1}, for any $p>1$ we can find $c_4>0$ independent of $\mu$ if $\mu>1$ such that
    \begin{align} \label{v.4}
        \int_\Omega v^{p-2}|\nabla v|^2 \geq c_4 \int_\Omega v^{p+\frac{2p_0}{n}}-c_4.
    \end{align}
    Multiplying the second equation of \eqref{1} by $p v^{p-1}$, using \eqref{v.4} and applying Young's inequality with some $\epsilon>0$, we obtain 
    \begin{align}\label{v.5'}
        \frac{d}{dt}\int_\Omega v^p + p \int_\Omega v^p &= -p(p-1)\int_\Omega v^{p-2}|\nabla v|^2 + \int_\Omega v^{p-1}uw \notag \\
        &\leq -c_5\int_\Omega v^{p+\frac{2p_0}{n}} + \int_\Omega v^{p-1}uw +c_5 \notag \\
        &\leq -c_5\int_\Omega v^{p+\frac{2p_0}{n}} + c_6\int_\Omega u^{\alpha+1}+ \epsilon \int_\Omega (v^{p-1}w)^{\frac{\alpha+1}{\alpha}} +c_5  \qquad \text{for all }t\in (0,T_{\rm max}),
    \end{align}
    where $c_5>0$ independent of $\mu$ if $\mu>1$ and $c_6=c_6(\epsilon)>0$ independent of $\mu$. Employing Young's inequality and noting from \eqref{v.1'} that 
    \begin{align*}
        \left (p- \frac{2q}{n+2q} \right ) \cdot \frac{\alpha+1}{\alpha} = p+ \frac{2p_0}{n},
    \end{align*}
     it follows that 
    \begin{align} \label{v.5}
         \int_\Omega (v^{p-1}w)^{\frac{\alpha+1}{\alpha}} &\leq \int_\Omega v^{ \left (p- \frac{2q}{n+2q} \right ) \cdot \frac{\alpha+1}{\alpha} } + \int_\Omega w^{\frac{n+2q}{n} \cdot \frac{\alpha+1}{\alpha} \cdot \left ( p- \frac{2q}{n+2q} \right)} \notag \\
         &\leq \int_\Omega v^{p+\frac{2p_0}{n}} + \int_\Omega w^{\frac{n+2q}{n} \left ( p+\frac{2p_0}{n} \right )}.
    \end{align}
    Setting 
    \begin{align*}
        y(t) = \int_\Omega v^p(\cdot,t) + \lambda \int_\Omega w^r(\cdot,t) \qquad \text{for all }t \in (0,T_{\rm max}),
    \end{align*}
    with $\lambda = \frac{c_5}{2c_3}$. Let fix $\epsilon= \min \left \{ \frac{c_5}{2}, \frac{c_2c_5}{2c_3} \right \}$.
    Collecting \eqref{v.3}, \eqref{v.5'} and \eqref{v.5} and using the identity \eqref{v.2}, we arrive at 
    \begin{align*}
        y'(t) +y(t) &\leq \left ( -c_5+\lambda c_3 + \epsilon \right  ) \int_\Omega v^{p+\frac{2p_0}{n}} + \left ( -\lambda c_2 + \epsilon \right ) \int_\Omega w^{\frac{n+2q}{n} \cdot \left ( p+\frac{2p_0}{n} \right )} +c_6 \int_\Omega u^{\alpha+1} +c_7 \notag \\
        &\leq  c_6 \int_\Omega u^{\alpha+1} +c_7 \qquad \text{for all }t\in (0,T_{\rm max}),
    \end{align*}
    with $c_7=c_5+\lambda c_3$. This, together with Lemma \ref{L1}[\eqref{L1-1}] and Lemma \ref{ODI} implies that  
    \begin{align*}
        \int_\Omega v^p(\cdot,t) + \lambda \int_\Omega w^r(\cdot,t) \leq c_8\qquad \text{for all }t\in (0,T_{\rm max}),
    \end{align*}
    for some $c_8>0$ independent of $\mu $ since $\mu>1$. The proof is now complete.
\end{proof}

We are now in a position to prove the central result of this section: a uniform-in-time $L^\infty$ bound for $w$.

\begin{lemma} \label{wL}
    Let $\gamma=1$, $\alpha > \frac{n-2}{4}$ with $n \geq 2$, $k=0$, $\chi \in \mathbb{R}$ and $\mu>0$. There exists $p> \frac{n}{2}$ and $C_1>0$  such that 
    \begin{align} \label{wL-1}
        \int_\Omega v^p(\cdot,t ) \leq C_1 \qquad \text{for all }t\in (0,T_{\rm max}).
    \end{align}
    As a consequence, there exists $C_2>0$ fulfilling
    \begin{align} \label{wL-2}
        \left \| w(\cdot,t) \right \|_{L^\infty(\Omega)} \leq C_2 \qquad \text{for all }t\in (0,T_{\rm max}).
    \end{align}
    Moreover, if $\mu>1$ then the constants $C_1$ and $C_2$ are independent of $\mu$.
\end{lemma}
\begin{proof}
  From Lemma \ref{L1}, we have that 
  \begin{align*}
      \int_\Omega v(\cdot,t) \leq c_0 \qquad \text{for all }t\in (0,T_{\rm max})
  \end{align*}
  where $c_0>0$ independent of $\mu$. In view of Lemma \ref{v}, it follows that 
  \begin{align*}
      \int_\Omega v^{p_1}(\cdot,t) \leq c_1  \qquad \text{for all }t\in (0,T_{\rm max}),
  \end{align*}
  where $p_1= \frac{4\alpha+2}{2n}+\frac{1}{2} \in \left (1, \frac{4\alpha+2}{n} \right )$ and $c_1>0$ independent of $\mu$ if $\mu>1$. Letting $p_l= \left ( \frac{4\alpha+2}{2n}+\frac{1}{2} \right )^l $ where $l\geq 1$.  By using induction and Lemma \ref{v}, we deduce that if $p_{l-1} \leq \frac{n}{2}$ and 
  \begin{align*}
        \int_\Omega v^{p_{l-1}}(\cdot,t) \leq c_{l-1} \qquad \text{for all }t\in (0,T_{\rm max}).
    \end{align*}
    with $c_{l-1}>0$ independent of $\mu $ if $\mu>1$, then
    \begin{align*}
        \int_\Omega v^{p_{l}}(\cdot,t) \leq c_{l} \qquad \text{for all }t\in (0,T_{\rm max}).
    \end{align*}
    with $c_{l}>0$ independent of $\mu $ if $\mu>1$. Since $p_l \to \infty$ as $l \to \infty$, there exists $l_0 \in \mathbb{N}$ such that $p_{l_0-1}\leq  \frac{n}{2}$ and $p_{l_0}> \frac{n}{2}$ satisfying  
      \begin{align*}
        \int_\Omega v^{p_{l_0}}(\cdot,t) \leq c_{l_0} \qquad \text{for all }t\in (0,T_{\rm max}),
    \end{align*}
    where $c_{l_0}>0$ independent of $\mu $ if $\mu>1$. Therefore, \eqref{wL-1} is proved. In view of Lemma \ref{C52.Para-Reg} and \eqref{wL-1}, we  obtain 
    \begin{align*}
         \left \|w(\cdot,t) \right \|_{L^\infty(\Omega)} \leq C \qquad \text{for all }t\in (0,T_{\rm max}),
    \end{align*}
   where $C>0$ independent of $\mu$ if $\mu>1$. The proof is now complete. 
    
\end{proof}

\section{Global boundedness for fully parabolic system} \label{S4}

In this section, we establish several essential estimates---including an $L\ln L$ estimate in two spatial dimensions and $L^p$ estimates for $p>1$ in arbitrary dimensions---that are needed for the proof of the first main theorem. We begin by deriving an $L\ln L$ bound for $u$ and an $L^2$ bound for $\nabla v$, which are presented in the following lemma.

\begin{lemma} \label{LlnL}
    Let $n =2$, $\gamma=1$, $\alpha=1$, $k\in [0,1)$, $\chi \in \mathbb{R}$ and $\mu>0$. There exists $C>0$ such that 
    \begin{align}
        \int_\Omega u(\cdot,t) \ln (u(\cdot,t) +e)+ \int_\Omega |\nabla v(\cdot,t)|^2 \leq C   \qquad \text{for all }t\in (0,T_{\rm max}).
    \end{align}
\end{lemma}
\begin{proof}
    Multiplying the first equation of \eqref{1} by $\ln u+1$ and integrating by parts yields
    \begin{align} \label{LlnL.1}
        \frac{d}{dt}\int_\Omega u \ln u &= \int_\Omega (\ln u+1) \left ( \Delta u - \chi \nabla \cdot (u \nabla v) +\kappa-u-uw - \frac{\mu u^2}{\ln^k(u+e)} \right )\notag\\
        &= - \int_\Omega \frac{|\nabla u|^2}{u} -\chi \int_\Omega u \Delta v - \int_\Omega u \ln u -\int_\Omega u -\int_\Omega w u \ln u \notag \\
        &\quad +\int_\Omega \kappa (\ln u+1)-\int_\Omega uw -\mu \int_\Omega \frac{u^2(\ln u+1)}{\ln^k(u+e)} \notag \\
        &\leq - \int_\Omega \frac{|\nabla u|^2}{u} -\chi \int_\Omega u \Delta v - \int_\Omega u \ln u - \mu \int_\Omega \frac{u^2 \ln u}{\ln^k(u+e)} +c_1 \qquad \text{for all }t\in (0,T_{\rm max}),
    \end{align}
    where $c_1>0$ and the last inequality holds due to \eqref{L1-1} and
    \begin{align*} 
         -\int_\Omega w u \ln u  \leq \frac{1}{e} \int_\Omega w,
    \end{align*}
    and 
    \begin{align*}
        \int_\Omega \kappa (\ln u+1) \leq \kappa_* \int_\Omega u.
    \end{align*}
     One can verify that there exists $c_2>0$ such that 
    \begin{align} \label{LlnL.3}
        - \mu \int_\Omega \frac{u^2 \ln u}{\ln^k(u+e)} \leq -\frac{\mu }{2} \int_\Omega u^2 \ln^{1-k}(u+e) +c_2.
    \end{align}
    Indeed, there exist $A_1>0$, $A_2>0$ and $A_3>0$ such that 
    \begin{align} \label{LlnL.4}
         \int_{\left\{ u \leq e \right \}} \frac{u^2 \ln u}{\ln^k(u+e)} \geq -A_1
    \end{align}
    and 
    \begin{align} \label{LlnL.5}
        \int_{ \left \{ u \leq e \right \}} u^2 \ln^{1-k}(u+e) \leq A_2,
    \end{align}
    as well as 
    \begin{align} \label{LlnL.6}
        \int_{ \left \{ u>e \right \}} \frac{u^2 \ln u}{\ln^k(u+e)} &\geq  \int_{ \left \{ u > e \right \}} u^2 \ln^{1-k}(u+e) -\ln 2 \int_{ \left \{ u > e \right \}}\frac{u^2}{\ln^k(u+e)} \notag \\
        &\geq \frac{1}{2}\int_{ \left \{ u > e \right \}} u^2 \ln^{1-k}(u+e)+A_3.
    \end{align} 
    Combining \eqref{LlnL.4}, \eqref{LlnL.5} and \eqref{LlnL.6}, we obtain \eqref{LlnL.3}. By direct calculations, we have
    \begin{align}\label{LlnL.7}
        \frac{1}{2} \frac{d}{dt} \int_\Omega |\nabla v|^2 &= \int_\Omega \nabla v \cdot \nabla (\Delta v -v  +uw) \notag \\
        &=-\int_\Omega |\Delta v|^2 -\int_\Omega |\nabla v|^2 - \int uw \Delta v.
    \end{align}
   From Lemma \ref{L1}[\eqref{L1-2}], we deduce that there exists $C>0$ such that 
   \begin{align*}
      \int_{t}^{t+\tau } \int_\Omega u^{\frac{3}{2}} \leq C  \qquad \text{for all }t\in (0,T_{\rm max} -\tau ),
   \end{align*}
   where $\tau = \min \left \{1, \frac{T_{\rm max}}{2} \right \}$. Applying this and Lemma \ref{wL}[\eqref{wL-2}] with $\alpha=\frac{1}{2}$, we infer that  
   \begin{align*}
       \sup_{t \in (0,T_{\rm max})} \left \|w(\cdot,t) \right \|_{L^\infty(\Omega)}<\infty.
   \end{align*}
   Using this and Young's inequality, it follows that 
   \begin{align}\label{LlnL.8}
       - \int uw \Delta v -\chi \int_\Omega u \Delta v &\leq c_3 \int_\Omega u |\Delta v| \notag \\
       &\leq \frac{1}{2}\int_\Omega |\Delta v|^2 + \frac{c_3^2}{2} \int_\Omega u^2, 
   \end{align}
   where $c_3= |\chi|+ \sup_{t \in (0,T_{\rm max})} \left \|w(\cdot,t) \right \|_{L^\infty(\Omega)}$. Collecting \eqref{LlnL.1}, \eqref{LlnL.3}, \eqref{LlnL.7} and \eqref{LlnL.8}, we arrive at 
   \begin{align*}
       \frac{d}{dt} \left \{  \int_\Omega u \ln u + \frac{1}{2} \int_\Omega |\nabla v|^2 \right \}+ \int_\Omega u \ln u + \frac{1}{2} \int_\Omega |\nabla v|^2 &\leq c_1+c_2+\frac{c_3^2}{2}\int_\Omega u^2 -\frac{\mu}{2}\int_\Omega u^2 \ln^{1-k}(u+e) \notag \\
       &\leq c_4 \qquad \text{for all }t\in (0,T_{\rm max}),
   \end{align*}
   where $c_4>0$. Applying Gronwall's inequality to this, we deduce that 
   \begin{align}\label{LlnL.9'}
       \int_\Omega u(\cdot,t ) \ln u(\cdot,t ) +\frac{1}{2}\int_\Omega |\nabla v(\cdot,t)|^2 \leq  c_5:= \max \left \{ c_4, \int_\Omega u_0 \ln u_0 +\frac{1}{2}\int_\Omega |\nabla v_0|^2 \right \}.
   \end{align}
   It is clear that 
   \begin{align}\label{LlnL.9}
       \int_\Omega u \ln (u+e) &= \int_{\left \{ u \leq e \right \}  } u \ln (u+e)+ \int_{\left \{ u > e \right \}  } u \ln (u+e) \notag \\
       &\leq e\ln(2e)|\Omega| +  \int_{\left \{ u > e \right \}  } u \ln (2u) \notag \\
       &\leq \int_{\left \{ u > e \right \}  } u \ln u + c_6,
   \end{align}
   where $c_6>0$. Moreover, we have
   \begin{align}\label{LlnL.10}
       \int_{ \left \{ u \leq e \right \}} u \ln u \geq -\frac{|\Omega|}{e}.
   \end{align}
    Combining \eqref{LlnL.9'}, \eqref{LlnL.9} and \eqref{LlnL.10}, it follows that 
    \begin{align*}
        \int_\Omega u(\cdot, t) \ln (u(\cdot,t)+e) &\leq \int_\Omega u(\cdot, t) \ln u(\cdot, t) + c_6+ \frac{|\Omega|}{e} \notag \\
        &\leq c_7 \qquad \text{for all }t\in (0,T_{\rm max}),
    \end{align*}
    where $c_7= c_5+c_6+\frac{|\Omega|}{e}$. The proof is now complete.
\end{proof}

The next result furnishes an elementary yet useful estimate, which will subsequently be employed in deriving $L^p$ bounds for $p>1$ in two spatial dimensions.
\begin{lemma} \label{C5.l}
 Let $n\geq 2$, $\gamma=1$, $\alpha>\frac{n-2}{4}$, $k=0$, $\chi \in \mathbb{R}$ and $\mu>0$. For any $p >1$, there exist positive constants $C_1$ and $C_2$ such that 
    \begin{align} \label{C5.l-1}
        \frac{d}{dt}\int_\Omega |\nabla v|^{2p}+ \int_\Omega |\nabla v|^{2p} \leq -\frac{p-1}{p}\int_\Omega \left | \nabla |\nabla v|^p \right |^2+ C_1 \int_\Omega u^2|\nabla v|^{2p-2} +C_2 \quad \text{for all }t\in (0,T_{\rm max}).
    \end{align}
   
\end{lemma}
\begin{proof}
Direct calculations show that 
\begin{align} \label{C5.l.1}
    \frac{1}{p} \frac{d}{dt}\int_\Omega |\nabla v|^{2p} &= 2\int_\Omega |\nabla v|^{2p-2} \nabla v \cdot \nabla \left ( \Delta v -  v + uw \right ) \notag \\
    &= \int_\Omega |\nabla v|^{2p-2} \left ( \Delta |\nabla v|^2- 2|D^2 v|^2 \right )- 2\int_\Omega |\nabla v|^{2p}+2 \int_\Omega |\nabla v|^{2p-2} \nabla v \cdot \nabla (uw) \notag \\
    &= -\frac{4(p-1)}{p^2} \int_\Omega  |\nabla |\nabla v|^p|^2 - 2 \int_\Omega |\nabla v|^{2p-2}|D^2 v|^2- 2 \int_\Omega |\nabla v|^{2p}\notag \\
    &\quad+2 \int_\Omega |\nabla v|^{2p-2} \nabla v \cdot \nabla (uw) + \int_{\partial \Omega} |\nabla v|^{2p-2}\frac{\partial |\nabla v|^2}{\partial \nu},
\end{align}
where we have used the identity that $\Delta |\nabla v|^2 = 2\nabla v \cdot \nabla \Delta v +2 |D^2 v|^2$. We apply Lemma 4.2 in \cite{Souplet-13} to obtain that 
   {\begin{align*}
       \frac{\partial |\nabla v|^2}{\partial \nu } \leq c_1 |\nabla v|^2 \quad \text{on } \partial \Omega,
   \end{align*}}
   for some $c_1>0$. This, together with Trace Sobolev's embedding theorem, $W^{1,1}(\Omega) \to L^1(\partial \Omega)$, {and} Young's inequality entails that
   \begin{align}\label{C5.l.2}
       \int_{\partial \Omega}|\nabla v|^{2(p-1)} \frac{\partial |\nabla v|^2}{\partial \nu} &\leq c_1  \int_{\partial \Omega}|\nabla v|^{2p} \notag \\
       &\leq c_2 \int_\Omega |\nabla v|^{2p}+ c_2 \int_\Omega |\nabla v|^{2p-2} |\nabla |\nabla v|^2| \notag \\
       &\leq \frac{p-1}{p^2}\int_\Omega |\nabla |\nabla v|^p|^2+c_3  \int_\Omega |\nabla v|^{2p},
   \end{align}
   where $c_2>0$ and $c_3>0$. From Lemma \ref{LlnL}, we have that $\int_\Omega |\nabla v|^2$ is bounded in $(0,T_{\rm max})$. This, in conjunction with Gagliardo–Nirenberg interpolation inequality and Young's inequality {implies} that 
   \begin{align}\label{C5.l.3}
       \left (c_3+\frac{1}{p} \right  )  \int_\Omega |\nabla v|^{2p} \leq \frac{p-1}{p^2}\int_\Omega |\nabla |\nabla v|^p|^2+c_4,
   \end{align}
   for some $c_4>0$. By using integration by parts and Lemma \ref{wL}[\eqref{wL-2}], we derive that
   \begin{align}\label{C5.l.4}
       2 \int_\Omega |\nabla v|^{2p-2} \nabla v \cdot \nabla (uw)&= -2 \int_\Omega uw |\nabla v|^{2p-2}\Delta v  - \frac{4(p-1) }{p} \int_\Omega  uw|\nabla v|^{p-2} \nabla v \cdot \nabla |\nabla v|^p  \notag \\
       &\leq \frac{2}{n} \int_\Omega |\nabla v|^{2p-2}|\Delta v|^2 + \frac{p-1}{p^2}\int_\Omega |\nabla |\nabla v|^p|^2+c_5 \int_\Omega u^2w^2|\nabla v|^{2p-2} \notag \\
       &\leq 2\int_\Omega |\nabla v|^{2p-2} |D^2 v|^2 + \frac{p-1}{p^2}\int_\Omega |\nabla |\nabla v|^p|^2+c_6 \int_\Omega u^2|\nabla v|^{2p-2}, 
   \end{align}
   where $c_5>0$ and $c_6= c_5 \sup_{t\in (0,T_{\rm max})} \left \| w(\cdot,t) \right \|_{L^\infty(\Omega)} $. Collecting from \eqref{C5.l.1} to \eqref{C5.l.4}, {it follows that}
   \begin{align*}
       \frac{d}{dt} \int_\Omega |\nabla v|^{2p}+ \int_\Omega |\nabla v|^{2p} + \frac{p-1}{p}\int_\Omega |\nabla |\nabla v|^p|^2 \leq c_6p \int_\Omega u^2|\nabla v|^{2p-2}+c_4p \quad \text{for all } t\in (0,T_{\rm max}), 
   \end{align*}
   which completes the proof.
\end{proof}

We now derive uniform-in-time $L^p$ bounds for $u$ and $|\nabla v|^2$ for any $p>1$ in two spatial dimensions. It is worth emphasizing that this result does not require any largeness assumption on $\mu$, thanks to the uniform-in-time $L\ln L$ bound for $u$ established in Lemma \ref{LlnL}..

\begin{lemma} \label{Lp''}
    Let $n=2$, $\gamma=1$, $\alpha=1$, $k \in [0,1)$, $\chi \in \mathbb{R}$ and $\mu>0$. For any $p>1$, there exists $C>0$ such that 
    \begin{align}
        \int_\Omega u^p(\cdot,t) +\int_\Omega |\nabla v(\cdot,t)|^{2p} \leq C \qquad \text{for all }t\in (0,T_{\rm max}). 
    \end{align}
\end{lemma}
\begin{proof}
    Testing the first equation of \eqref{1} by $pu^{p-1}$ with $p>1$ and applying Young's inequality, there exists $c_1>0$ such that 
    \begin{align}\label{Lp''.1}
       \frac{d}{dt} \int_\Omega u^p &= p\int_\Omega u^{p-1} \left ( \Delta u - \chi \nabla \cdot (u \nabla v) +\kappa-u-uw - \frac{\mu u^2}{\ln^k(u+e)} \right )\notag\\ 
       &= -p(p-1) \int_\Omega u^{p-2}|\nabla u|^2 +\chi p(p-1)\int_\Omega u^{p-1} \nabla u \cdot \nabla v + p\int_\Omega u^{p-1} \kappa  \notag \\
       &\quad- p\int_\Omega u^p - p\int_\Omega u^pw - \mu p \int_\Omega \frac{u^{p+1}}{\ln^k(u+e)} \notag \\
       &\leq -\frac{p(p-1)}{2}\int_\Omega u^{p-2}|\nabla u|^2 + \frac{\chi^2p(p-1)}{2} \int_\Omega u^p |\nabla v|^2 \notag \\
       &\quad+ p\kappa_* \int_\Omega u^{p-1}   - p\int_\Omega u^p \notag \\
       &\leq - \frac{2(p-1)}{p} \int_\Omega |\nabla u^{\frac{p}{2}}|^2+ \frac{\chi^2 p(p-1)}{2} \int_\Omega u^p |\nabla v|^2 -\int_\Omega u^p +c_1 \qquad \text{for all }t\in (0,T_{\rm max}).
    \end{align}
    From Lemma \ref{C5.l}, one can find $c_2>0$ and $c_3>0$ such that
    \begin{align}\label{Lp''.2}
        \frac{d}{dt} \int_\Omega |\nabla v|^{2p}+\int_\Omega |\nabla v|^{2p} + \frac{p-1}{p} \int_\Omega |\nabla |\nabla v|^p|^2 \leq c_2 \int_\Omega u^2 |\nabla v|^{2p-2} +c_3 \qquad \text{for all }t\in (0,T_{\rm max}).
    \end{align}
    {By applying Lemma \ref{LlnL} and Lemma \ref{GN}, we obtain}
    \begin{align}\label{Lp''.10}
        \int_\Omega |\nabla v|^{2p+2} &\leq c_{4} \int_\Omega |\nabla |\nabla v|^p|^2 \int_\Omega |\nabla v|^2 +c_{4} \left (\int_\Omega |\nabla v|^2 \right )^2 \notag \\
        &\leq c_{5}\int_\Omega |\nabla |\nabla v|^p|^2 +c_{6},
    \end{align}
    where $c_{4}$, $c_{5}$, and $c_{6}$ are positive constants. In light of Young's inequality and \eqref{Lp''.10}, it follows that 
    \begin{align}\label{Lp''.11}
        \frac{\chi^2 p(p-1)}{2}\int_\Omega u^p |\nabla v|^2 + c_{2}\int_\Omega u^2 |\nabla v|^{2p-2} &\leq \frac{p-1}{pc_{5}} \int_\Omega |\nabla v|^{2p+2} +c_{7}\int_\Omega u^{p+1} \notag \\
        &\leq \frac{p-1}{p}\int_\Omega |\nabla |\nabla v|^p|^2 +c_{7}\int_\Omega u^{p+1}+c_{8},
    \end{align}
    where $c_7>0$ and $c_8>0$. From Lemma \ref{LlnL}, there exists $c_{9}>0$ such that 
     \begin{align*}
         \int_\Omega u(\cdot,t) \ln(u(\cdot,t)+e) \leq c_{9} \qquad \text{for all }t\in (0,T_{\rm max}).
     \end{align*}
     Making use of Lemma \ref{C52.ILGN} with $\varepsilon = \frac{2(p-1)}{pc_{7}c_{9}}$, {we find that} 
     \begin{align} \label{Lp''.12}
         c_{7} \int_\Omega u^{p+1} &\leq  c_{7}\varepsilon \int_\Omega |\nabla u^\frac{p}{2}|^2 \int_\Omega u \ln (u+e) + c_{7}\varepsilon \left ( \int_\Omega u \right )^p \int_\Omega u \ln (u+e) +c_{10} \notag \\
         &\leq \frac{2(p-1)}{p} \int_\Omega |\nabla u^\frac{p}{2}|^2 +c_{11},
     \end{align}
     where $c_{10}>0$ and $c_{11}>0$. Now, collecting \eqref{Lp''.1}, \eqref{Lp''.2}, \eqref{Lp''.11}, \eqref{Lp''.12}, we arrive at 
     \begin{align}\label{Lp''.13}
         \frac{d}{dt} \left \{ \int_\Omega u^p+ \int_\Omega |\nabla v|^{2p} \right \}+ \int_\Omega u^p+ \int_\Omega |\nabla v|^{2p} \leq c_{12} \qquad \text{for all }t\in (0,T_{\rm max}),
     \end{align}
     where $c_{12}=c_1+c_3+c_8+c_{11}$. Finally, we apply Gronwall's inequality to \eqref{Lp''.13} to obtain that 
     \begin{align}
         \int_\Omega u^p(\cdot,t)+ \int_\Omega |\nabla v(\cdot,t)|^{2p} \leq \max \left \{ c_{12}, \int_\Omega u_0^p + \int_\Omega |\nabla v_0|^{2p} \right \} \qquad \text{for all }t\in (0,T_{\rm max}),
     \end{align}
    which completes the proof.
\end{proof}
    
In higher dimensions, we next show that the solution component $u$ is globally bounded in $L^p$ under certain conditions on $\alpha$ and $\mu$; this is stated in the following lemma. It is worth mentioning that this result demonstrates that quadratic logistic damping can prevent blow-up in three spatial dimensions, which is the physically relevant case.
\begin{lemma} \label{Lp}
    Let $p>1$, $\gamma=1$ and $\chi \in \mathbb{R}$. If $\alpha=1$, $k=0$  and $3\leq n \leq 5$ then there exists $\mu_*= \mu_*(p)>0$ such that if $\mu>\mu_*$ then there exists $C>0$ such that
    \begin{align} \label{Lp-1}
        \int_\Omega u^p(\cdot,t) \leq C \qquad \text{for all }t\in (0,T_{\rm max}).
    \end{align}
    Moreover, if $k=0$, $\alpha>\frac{n-2}{4}$ with $n \geq 6$ and $\mu>0$, then \eqref{Lp-1} also holds.
\end{lemma}

\begin{proof}
    Multiplying  the first equation of \eqref{1} by $u^{p-1}$ with $p>1$ and integrating by parts yields
    \begin{align} \label{Lp.1}
         \frac{d}{dt} \int_\Omega u^p &= -(p-1)p\int_\Omega u^{p-2} |\nabla u|^2 - \chi(p-1) \int_\Omega u^p \Delta v + p\int_\Omega u^{p-1}\kappa - p\int_\Omega u^p w - \mu p \int_\Omega u^{p+\alpha} \notag \\
         &\leq -\chi(p-1) \int_\Omega u^p \Delta v + p\int_\Omega u^{p-1}\kappa  - \mu p \int_\Omega u^{p+\alpha}  \qquad \text{for all }t\in (0,T_{\rm max}).
    \end{align}
    By using Young's inequality, we obtain 
    \begin{align}\label{Lp.2}
        \frac{p+1}{2} \int_\Omega u^p + p\int_\Omega u^{p-1}\kappa &\leq  \frac{p+1}{2} \int_\Omega u^p  + p \kappa_* \int_\Omega u^{p-1} \notag \\
        &\leq \frac{\mu p}{4} \int_\Omega u^{p+1}+c_1,
    \end{align}
    where $c_1>0$ and 
    \begin{align}\label{Lp.3}
        -\chi(p-1) \int_\Omega u^p \Delta v \leq \frac{\mu p}{4}\int_\Omega u^{p+1}+ \frac{c_2}{\mu^p} \int_\Omega |\Delta v|^{p+1}. 
    \end{align}
    $c_2= \frac{(p-1)^{p+1}}{p+1} \left ( \frac{4}{p}\right )^p$. Collecting \eqref{Lp.1}, \eqref{Lp.2}, and \eqref{Lp.3}, we obtain that
    \begin{align}\label{Lp.4}
         \frac{d}{dt} \int_\Omega u^p +  \frac{p+1}{2} \int_\Omega u^p \leq \frac{c_2}{\mu^p} \int_\Omega |\Delta v|^{p+1} +\frac{\mu p}{2}\int_\Omega u^{p+1} - \mu p \int_\Omega u^{p+\alpha} +c_1  \qquad \text{for all }t\in (0,T_{\rm max}).
    \end{align}
    Multiplying both sides of \eqref{Lp.4} by $e^{\frac{p+1}{2}t}$ and integrating from $t_0:= \min \left \{1, \frac{T_{\rm max}}{2} \right \}$ to $t$, we arrive at 
    \begin{align}\label{Lp.5}
        \int_\Omega u^p(\cdot,t) &\leq \frac{c_2}{\mu^p} \int_{t_0}^t e^{-\frac{(p+1)(t-s)}{2}} \int_\Omega |\Delta v(\cdot,s)|^{p+1} + \frac{\mu p}{2} \int_{t_0}^t e^{-\frac{(p+1)(t-s)}{2}} \int_\Omega u^{p+1}(\cdot,s)  \notag \\
        &\quad - \mu p  \int_{t_0}^t e^{-\frac{(p+1)(t-s)}{2}} \int_\Omega u^{p+\alpha}(\cdot,s) + \int_\Omega u^p(\cdot,t_0)+ \frac{2c_1}{p+1}. 
    \end{align}
    In view of Lemma \ref{l1} and Lemma \ref{wL}[\eqref{wL-2}], we can find $c_3=c_3(p,n,\Omega)>0$ such that  
    \begin{align}\label{Lp.6}
        \frac{c_2}{\mu^p} \int_{t_0}^t e^{-\frac{(p+1)(t-s)}{2}} \int_\Omega |\Delta v(\cdot,s)|^{p+1} &\leq  \frac{c_3}{\mu^p} \int_{t_0}^t e^{-\frac{(p+1)(t-s)}{2}} \int_\Omega u^{p+1}(\cdot,s)w^{p+1}(\cdot,s)  \notag \\
        &\quad+ \frac{c_3}{\mu^p} e^{-\frac{(p+1)(t-t_0)}{2}} \int_\Omega |\Delta v(\cdot,t_0)|^{p+1} \notag \\
        &\leq \frac{c_4}{\mu^p}\int_{t_0}^t e^{-\frac{(p+1)(t-s)}{2}} \int_\Omega u^{p+1}(\cdot,s)  \notag \\
        &\quad+ \frac{c_3}{\mu^p} e^{-\frac{(p+1)(t-t_0)}{2}} \int_\Omega |\Delta v(\cdot,t_0)|^{p+1},
    \end{align}
    where $c_4= c_3 \sup_{t\in (0,T_{\rm max})} \left \| w(\cdot,t)  \right \|^{p+1}_{L^\infty(\Omega)} $ is independent of $\mu $ if $\mu>1$. Combining \eqref{Lp.5} and \eqref{Lp.6}, we obtain 
    \begin{align}\label{Lp.7}
        \int_\Omega u^p(\cdot,t) &\leq \left ( \frac{\mu p}{2} +\frac{c_4}{\mu^p} \right )\int_{t_0}^t e^{-\frac{(p+1)(t-s)}{2}} \int_\Omega u^{p+1}(\cdot,s) -\mu p  \int_{t_0}^t e^{-\frac{(p+1)(t-s)}{2}} \int_\Omega u^{p+\alpha}(\cdot,s) +c_5,
    \end{align}
   {for all } $t\in (0,T_{\rm max})$,  where $c_5 =\frac{c_3}{\mu^p} e^{-\frac{(p+1)(t-t_0)}{2}} \int_\Omega |\Delta v(\cdot,t_0)|^{p+1} +\int_\Omega u^p(\cdot,t_0) +\frac{2c_1}{p+1} $. In case $\alpha=1$ and $3 \leq n \leq 5$, we choose $\mu_* := \max \left \{1, \left ( \frac{2c_4}{p} \right )^{\frac{1}{p+1}} \right \}$, then for any $\mu>\mu_*$ it follows that 
    \begin{align}\label{Lp.8}
         \int_\Omega u^p(\cdot,t) &\leq \left ( \frac{c_4}{\mu^p} - \frac{\mu p}{2} \right ) \int_{t_0}^t e^{-\frac{(p+1)(t-s)}{2}} \int_\Omega u^{p+1}(\cdot,s) +c_5 \notag\\
         &\leq c_5 \qquad \text{for all }t\in (0,T_{\rm max}).
    \end{align}
    In case $\alpha> \frac{n-2}{4}$ with $n \geq 6$, we find that $\alpha>1$. By applying Young's inequality to \eqref{Lp.7}, we can find $c_6>0$ such that 
    \begin{align*}
         \int_\Omega u^p(\cdot,t) &\leq c_6\qquad \text{for all }t\in (0,T_{\rm max}).
    \end{align*}
    This, together with \eqref{Lp.8} completes the proof of the lemma.    
\end{proof}

We are now ready to prove our first main result.
\begin{proof}[Proof of Theorem \ref{thm}]
    In case $3\leq n \leq 5$, $\alpha=1$ and $k =0$, from Lemma \ref{Lp}, there exists $\mu_0>0$ such that for any $\mu>\mu_0$, the following holds
    \begin{align} \label{prf.1}
        \int_\Omega u^{n+1}(\cdot,t) \leq c_1 \qquad \text{for all }t\in (0,T_{\rm max}),
    \end{align}
    where $c_1>0$.
    In case $n \geq 6$, $\alpha> \frac{n-2}{4}$ and $k=0$ or $n=2$, $\alpha=1$, and $k \in [0,1)$, \eqref{prf.1} holds for any $\mu>0$ thanks to Lemma \ref{Lp''} and Lemma \ref{Lp}. It is followed from Lemma \ref{wL}[\eqref{wL-2}] that 
    \begin{align}\label{prf.2}
         \left \| w(\cdot,t) \right \|_{L^{ \infty}(\Omega)} \leq c_2 \qquad \text{for all }t\in (0,T_{\rm max})
    \end{align}
    where $c_2>0$. This, together with \eqref{prf.1} implies that 
     \begin{align*} 
        \int_\Omega u^{n+1}(\cdot,t)w^{n+1}(\cdot,t) \leq c_3 \qquad \text{for all }t\in (0,T_{\rm max}),
    \end{align*}
    where $c_2=c_1 c_2^{n+1}$. This, in conjunction with Lemma \ref{C52.Para-Reg} implies that 
    \begin{align} \label{prf.3}
        \left \| v(\cdot,t) \right \|_{W^{1, \infty}(\Omega)} \leq c_4 \qquad \text{for all }t\in (0,T_{\rm max}),
    \end{align}
    where $c_4>0$. Using the estimate \eqref{Lp''.1}, we obtain
    \begin{align*}
        \frac{d}{dt}\int_\Omega u^p + \int_\Omega u^p &\leq -\frac{2(p-1)}{p} \int_\Omega |\nabla u^{\frac{p}{2}}|^2 + \frac{\chi^2p(p-1)}{2} \int_\Omega u^p |\nabla v|^2 +p\kappa_* \int_\Omega u^{p-1} \notag \\
        &\leq -\frac{2(p-1)}{p} \int_\Omega |\nabla u^{\frac{p}{2}}|^2 + \frac{\chi^2c_4^2p(p-1)}{2} \int_\Omega u^p +p\kappa_* \int_\Omega u^{p-1},
    \end{align*}
    for all $t\in (0,T_{\rm max})$. Now, applying standard Moser iteration procedure as well as established in Lemma A.1 in \cite{Winkler-2011}, we obtain that 
    \begin{align} \label{prf.4}
         \left \| u(\cdot,t) \right \|_{L^{ \infty}(\Omega)} \leq c_5 \qquad \text{for all }t\in (0,T_{\rm max}),
    \end{align}
    where $c_5>0$. Collecting \eqref{prf.2}, \eqref{prf.3}, and \eqref{prf.4}, it follows that 
    \begin{align*}
         \left \| u(\cdot,t) \right \|_{L^{ \infty}(\Omega)} +  \left \| v(\cdot,t) \right \|_{W^{1, \infty}(\Omega)} + \left \| w(\cdot,t) \right \|_{L^{ \infty}(\Omega)} \leq c_6 \qquad \text{for all }t\in (0,T_{\rm max}),
    \end{align*}
    where $c_6=c_2+c_4+c_5$. This, together with the extensibility property of solutions \eqref{local-1} implies that $T_{\rm max} = \infty$ and
    \begin{align*}
        \sup_{t>0} \left \{ \left \| u(\cdot,t) \right \|_{L^{ \infty}(\Omega)} +  \left \| v(\cdot,t) \right \|_{W^{1, \infty}(\Omega)} + \left \| w(\cdot,t) \right \|_{L^{ \infty}(\Omega)} \right \} <\infty,
    \end{align*}
    which completes the proof.
    
\end{proof}

\section{Global boundedness for the parabolic-elliptic-parabolic system } \label{S5}

In this section, we establish the global existence and boundedness of solutions to the parabolic--elliptic--parabolic system \eqref{1}. We begin with the $L^p$ boundedness result presented in the following lemma.

\begin{lemma} \label{Lp'}
    Let $\gamma=0$, $\alpha>\frac{n+2}{2}$ with $n \geq 2$, $k =0$, $\chi \in \mathbb{R}$ and $\mu>0$. Then for any $p>1$, there exists $C>0$ such that 
    \begin{align} \label{Lp'-1}
        \int_\Omega u^p(\cdot,t) \leq C \qquad \text{for all }t\in (0,T_{\rm max}).
    \end{align}
    and 
    \begin{align} \label{Lp'-2}
        \int_\Omega w^p(\cdot,t) \leq C \qquad \text{for all }t\in (0,T_{\rm max}).
    \end{align}
    
\end{lemma}
\begin{proof}
    Multiplying the first equation of \eqref{1} by $u^{p-1}$ with $p>1$ and integrating by parts yields
    \begin{align} \label{Lp'.1}
        \frac{1}{p}\frac{d}{dt}\int_\Omega u^p &= \int_\Omega u^{p-1} \left ( \Delta u - \chi \nabla (u \nabla v)+\kappa -u - uw- \mu u^{1+\alpha} \right ) \notag \\
        &=-(p-1) \int_\Omega u^{p-2}|\nabla u|^2-\frac{\chi(p-1) }{p}\int_\Omega u^p \Delta v + \int_\Omega \kappa u^{p-1}  \notag \\
        &\quad-\int_\Omega u^p - \int_\Omega u^p w -\mu \int_\Omega u^{p+\alpha}  \notag \\
        &\leq \frac{\chi(p-1) }{p}\int_\Omega u^{p+1} w -  \frac{\chi(p-1) }{p}\int_\Omega u^{p} v \notag \\
        &\quad+ \kappa_* \int_\Omega u^{p-1} - \int_\Omega u^p - \mu \int_\Omega u^{p+\alpha} \qquad \text{for all }t\in (0,T_{\rm max}).
    \end{align}
   \textbf{The case $\chi>0$.} In light of Young's inequality, there exists $c_1=c_1(p, \chi, \mu, \alpha)>0$ such that 
    \begin{align}\label{Lp'.2}
         \frac{\chi(p-1) }{p}\int_\Omega u^{p+1} w  &\leq \frac{\mu}{2} \int_\Omega u^{p+\alpha} + c_1 \int_\Omega w^{\frac{p+\alpha}{\alpha-1}}.
    \end{align}
    Let denote 
    \begin{align*}
        q= \frac{p+\alpha}{\alpha-1} - \frac{2}{n}
    \end{align*}
    and 
    \begin{align*}
        r= \frac{q+\frac{2}{n}}{\frac{2}{n}+1} = \frac{n(p+\alpha)}{(n+2)(\alpha-1)},
    \end{align*}
    then we see that $q>1$ since $\alpha> \frac{n+2}{2}$. In view of Lemma \ref{L1'} and Lemma \ref{w}, we can find $c_2>0$ and $c_3>0$ satisfying 
    \begin{align}\label{Lp'.3}
        c_2\frac{d}{dt}\int_\Omega w^q +c_2\int_\Omega w^q+2c_1 \int_\Omega w^{\frac{p+\alpha}{\alpha-1}} \leq c_3 \int_\Omega v^r + c_3\qquad \text{for all }t\in (0,T_{\rm max}).
    \end{align}
     Noting that the condition $\alpha>\frac{n+2}{2}$ implies that 
    \begin{align}\label{Lp'.4}
        \frac{r(p+\alpha)}{p+\alpha-r} < \frac{p+\alpha}{\alpha-1}.
    \end{align}
    Now, we test the second equation of \eqref{1} by $v^{r-1}$, use employ Young's inequality to obtain that 
    \begin{align}\label{Lp'.5}
        c_3 \int_\Omega v^r &=- c_3(r-1) \int_\Omega v^{r-2}|\nabla v|^2 + c_3\int_\Omega uw v^{r-1 } \notag \\
        &\leq c_3\int_\Omega uw v^{r-1 } \notag \\
        &\leq \frac{\mu}{2} \int_\Omega u^{p+\alpha}+c_4 \int_\Omega w^{\frac{r(p+\alpha)}{p+\alpha-r}} \notag \\
        &\leq \frac{\mu}{2} \int_\Omega u^{p+\alpha} +c_1 \int_\Omega w^{\frac{p+\alpha}{\alpha-1}} +c_5,
     \end{align} 
     where $c_4>0$ and $c_5>0$. By Young's inequality, there exists $c_6>0$ such that 
     \begin{align}\label{Lp'.6}
         \kappa_* \int_\Omega u^{p-1} \leq \frac{p-1}{p} \int_\Omega u^p+c_6.
     \end{align}
     Setting 
     \begin{align*}
          y(t) = \frac{1}{p}\int_\Omega u^p(\cdot,t) +c_2 \int_\Omega w^q(\cdot,t ) \qquad \text{for all }t\in (0,T_{\rm max}),
     \end{align*}
     and collecting \eqref{Lp'.1}, \eqref{Lp'.2}, \eqref{Lp'.3}, \eqref{Lp'.5} and \eqref{Lp'.6}, we arrive at 
     \begin{align*}
         y'(t) +y(t) \leq c_7  \qquad \text{for all }t\in (0,T_{\rm max}),
     \end{align*}
     where $c_7=c_3+c_5+c_6$. This, together with Gronwall inequality implies that 
     \begin{align}\label{Lp'.7}
        \frac{1}{p}\int_\Omega u^p(\cdot,t) +c_2 \int_\Omega w^q(\cdot,t ) \leq \max \left \{y(0),c_7 \right \}  \qquad \text{for all }t\in (0,T_{\rm max}).
     \end{align}
    \textbf{The case $\chi \leq 0$. } Employing Young's inequality, there exists $c_8>0$ such that 
    \begin{align}\label{Lp'.8}
        -\frac{\chi(p-1)}{p}\int_\Omega u^p v &\leq |\chi| \int_\Omega u^p v \notag \\
        &\leq \frac{\mu}{2} \int_\Omega u^{p+\alpha} + c_8 \int_\Omega v^{\frac{p+\alpha}{\alpha}}.
    \end{align}
    Setting 
    \begin{align*}
        l = \frac{p+\alpha}{\alpha} \qquad \text{and } \qquad s= \frac{l(p+\alpha)}{p+\alpha -l} - \frac{2}{n},
    \end{align*}
    then by using similar argument as in \eqref{Lp'.5}, we obtain
    \begin{align}\label{Lp'.9}
        (c_8+1)\int_\Omega v^l \leq \frac{\mu}{2} \int_\Omega u^{p+\alpha} +c_9 \int_\Omega w^{\frac{l(p+\alpha)}{p+\alpha-l}}.
    \end{align}
    Thanks to Lemma \ref{L1'}[\eqref{L1'.2}] and Lemma \ref{w}, there exist positive constants $c_{10}, c_{11}$ and $c_{12}$ such that 
    \begin{align}\label{Lp'.10}
        \frac{d}{dt}\int_\Omega w^s + \int_\Omega w^s + c_{10}\int_\Omega w^{s+\frac{2}{n}} \leq c_{11} \int_\Omega v^{\frac{n}{n+2}\left (s+\frac{2}{n} \right )}+c_{12} \qquad \text{for all }t\in (0,T_{\rm max}).
    \end{align}
    The condition $\alpha> \frac{n+2}{2}$ implies that 
    \begin{align*}
        \frac{n}{n+2}\left (s+\frac{2}{n} \right ) = \frac{n+2}{n} \cdot \frac{l(p+\alpha)}{p+\alpha-l} < l.
    \end{align*}
    This, together with \eqref{Lp'.10} and Young's inequality entails that 
    \begin{align}\label{Lp'.11}
        \lambda \frac{d}{dt}\int_\Omega w^s + \lambda \int_\Omega w^s+ c_9 \int_\Omega w^{\frac{l(p+\alpha)}{p+\alpha-l}} &\leq c_{11} \lambda \int_\Omega v^{\frac{n+2}{n} \cdot \frac{l(p+\alpha)}{p+\alpha-l}} +\lambda c_{12} \notag \\
        &\leq \int_\Omega v^l +c_{13}\qquad \text{for all }t\in (0,T_{\rm max}),
    \end{align}
    where $\lambda= \frac{c_9}{c_{10}}$ and $c_{13}>0$. Now, collecting \eqref{Lp'.1}, \eqref{Lp'.6}, \eqref{Lp'.8}, \eqref{Lp'.9} and \eqref{Lp'.11}, we obtain 
    \begin{align*}
        \frac{d}{dt} \left \{ \frac{1}{p}\int_\Omega u^p + \lambda \int_\Omega w^s \right \}+ \frac{1}{p}\int_\Omega u^p + \lambda \int_\Omega w^s \leq c_{14} \qquad \text{for all }t\in (0,T_{\rm max}),
    \end{align*}
    where $c_{14}= c_6+c_{13}$. Applying Gronwall inequality to this leads to
    \begin{align}\label{Lp'.12}
         \frac{1}{p}\int_\Omega u^p(\cdot,t) + \lambda \int_\Omega w^s(\cdot,t) \leq \max \left \{ c_{14}, \frac{1}{p}\int_\Omega u_0^p + \lambda \int_\Omega v_0^l \right \}\qquad \text{for all }t\in (0,T_{\rm max}).
    \end{align}
    Combining \eqref{Lp'.7} and \eqref{Lp'.12} and noting that  $p$ is arbitrary,  we deduce \eqref{Lp'-1} and \eqref{Lp'-2}. The proof is now complete.
\end{proof}
We are now in position to prove our second main result.
\begin{proof}[Proof of Theorem \ref{thm2}]
    From Lemma \eqref{Lp'}, it follows that 
    \begin{align*}
        \sup_{t\in (0,T_{\rm max})} \left \| u(\cdot,t) w(\cdot,t) \right \|_{L^p(\Omega)} < \infty,
    \end{align*}
    for some $p>n$. By standard elliptic regularity theory, we deduce that 
    \begin{align*}
        \sup_{t\in (0,T_{\rm max})} \left \| v(\cdot,t)  \right \|_{W^{2,p}(\Omega)} < \infty,
    \end{align*}
    which together with Sobolev embedding theorem entails that 
      \begin{align} \label{prfthm2.1}
        \sup_{t\in (0,T_{\rm max})} \left \| v(\cdot,t)  \right \|_{W^{1,\infty}(\Omega)} < \infty.
    \end{align}
    Employing standard maximum principle to the third equation of \eqref{1} yields that 
     \begin{align} \label{prfthm2.2}
        \sup_{t\in (0,T_{\rm max})} \left \| w(\cdot,t)  \right \|_{L^{\infty}(\Omega)} < \infty.
    \end{align}
   Thanks to \eqref{prfthm2.1} and \eqref{prfthm2.2}, we can perform a standard Moser iteration procedure as well-established in Lemma A.1 in \cite{Winkler-2011} to deduce that 
   \begin{align}\label{prfthm2.3}
       \sup_{t\in (0,T_{\rm max})} \left \| u(\cdot,t)  \right \|_{L^{\infty}(\Omega)} < \infty.
   \end{align}
   Finally, combining \eqref{prfthm2.1}, \eqref{prfthm2.2}, \eqref{prfthm2.3} and the extensibility property of solutions \eqref{local-1} yields that $T_{\rm max}=\infty$ and 
   \begin{align*}
       \sup_{t>0} \left \{\left \| u(\cdot,t)  \right \|_{L^{\infty}(\Omega)}+\left \| v(\cdot,t)  \right \|_{W^{1,\infty}(\Omega)}+  \left \| w(\cdot,t)  \right \|_{L^{\infty}(\Omega)}\right \}
<\infty,   \end{align*}
which completes the proof.
\end{proof}

     \paragraph{Data Availability}
 Data sharing not applicable to this article as no datasets were generated or analyzed during
the current study.
\section*{Declarations}
\paragraph{Conflict of Interest} The authors declare that they have no conflict of interest.
\paragraph{Acknowledgments} Minh Le was supported by the Hangzhou Postdoctoral Research Grant.

\end{document}